\documentclass{amsart}[12pt]

\usepackage{amssymb}

\newtheorem{theorem}{Theorem}[section]
\newtheorem{lemma}[theorem]{Lemma}
\newtheorem{trditev}[theorem]{Proposition}

\theoremstyle{definition}

\newtheorem{remark}[theorem]{Remark}

\numberwithin{equation}{section}

\def\max{\mathop{\rm max}\nolimits}

\def\Span{\mathop{\rm Span}\nolimits}

\parskip=\smallskipamount

\begin{document}
\title[]
{On regular Stein neighborhoods of a union of two totally real planes in $\mathbb{C}^2$}
\author{Tadej Star\v{c}i\v{c}}
\address{Faculty of Education, University of Ljubljana, Kardeljeva Plo\v{s}\v{c}ad 16, 1000 Lju\-blja\-na, Slovenia and 
Institute of Mathematics, Physics and Mechanics, Jadranska 19, 1000 Ljubljana, Slovenia}
\email{tadej.starcic@pef.uni-lj.si}
\date{April 7, 2016}


\keywords{Stein neighborhoods, totally real planes, strongly pseudoconvex domains, strong deformation retraction\\
\indent Research supported by grants P1-0291 and J1-5432 from ARRS, Republic of Slovenia.}

\begin{abstract}
In this paper we find regular Stein neighborhoods of a union of totally real planes
$M=(A+iI)\mathbb{R}^2$ and $N=\mathbb{R}^2$ in $\mathbb{C}^2$, provided that the entries of a real $2 \times 2$ matrix $A$ are sufficiently small. A key step in 
our proof is a local construction of a suitable function $\rho$ near the origin. The sublevel sets of $\rho$ are strongly Levi pseudoconvex and admit 
strong deformation retraction to $M\cup N$. 
\end{abstract}

\maketitle

\section{Introduction}

The class of Stein manifolds is one of the most important classes of complex manifolds. There are many characterizations of Stein 
manifolds (see Remmert \cite{Remmert1}, Grauert \cite{lit9} and Cartan \cite{Car}). Also many classical problems in complex analysis 
are solvable on Stein manifolds (see the monographs \cite{lit21} and \cite{lit6}). Therefore it is a very useful property for a 
subset of a manifold to have open Stein neighborhoods.

On the other hand one would also like to understand the topology or the homotopy type of such neighborhoods. Also approximation 
theorems can be obtained if neighborhoods have further suitable properties (see Cirka \cite{Cir69}). Interesting results in this 
direction for real surfaces immersed (or embedded) into a complex surface were given by Forstneri\v{c} \cite[Theorem 2.2]{lit1n} and 
Slapar \cite{lit2n}. If $\pi:S\to X$ is a smooth immersion of a closed real surface into a complex surface with finitely many special 
double points and only flat hyperbolic complex points, then $\pi(S)$ has a basis of {\em regular} Stein neighborhoods; these are 
open Stein neighborhoods which admit a strong deformation retraction to $\pi (S)$ (for the precise definition see Sect. \ref{baza}). 
The problem is to find a good plurisubharmonic function locally near every double point (see \cite{F92,lit1n,lit2n}) or hyperbolic 
complex point (see \cite{lit2n}). We add here that elliptic complex points prevent the surface from having a basis of Stein 
neighborhoods due to the existence of Bishop discs (see \cite{Bi}), while the surface is locally polynomially convex at hyperbolic 
points by a result of Forstneri\v {c} and Stout (see \cite{FS}).

In this paper we consider a union of two totally real planes $M$ and $N$ in $\mathbb{C}^2$ with $M\cap N=\{ 0\}$. Every such union 
is complex-linearly equivalent to $\mathbb{R}^2\cup M(A)$, where $M(A)$ is the real span of the columns of the matrix $A+iI$. 
Moreover, $A$ is a real matrix determined up to real conjugacy and such that $A-iI$ is invertible. By a result of Weinstock (see \cite{Wein}) 
each compact subset of $\mathbb{R}^2\cup M(A)$ is polynomially convex if and only if $A$ has no purely imaginary eigenvalue of modulus 
greater than one. For matrices $A$ that satisfy this condition it is then reasonable to try to find regular Stein neighborhoods 
for $\mathbb{R}^2\cup M(A)$. If $A=0$ the situation near the origin coincides with the special double point of immersed real surface 
in complex surface mentioned above. When $A$ is diagonalizable over $\mathbb{R}$ with $\textrm{Trace}(A)=0$, a regular Stein neighborhood 
basis has been constructed by Slapar (see \cite[Proposition 3]{lit2n}).

In Sect. \ref{baza} we prove that regular Stein neighborhoods of $\mathbb{R}^2\cup M(A)$ in $\mathbb{C}^2$ can be constructed, if the entries 
of $A$ are sufficiently small. An important step in our proof is a local construction of a suitable function $\rho$ near the origin, depending 
smoothly on the entries of $A$. Furthermore, $\rho$ is strictly plurisubharmonic in complex tangent direction to its sublevel sets, and such 
that the sublevel sets shrink down to $M\cup N$. The Levi form of $\rho$ is a homogeneous polynomial of high degree and it is difficult to 
control its sign for bigger entries of $A$. It would also be interesting to generalize the construction to the case of a union of two totally 
real subspaces of maximal dimension in $\mathbb{C}^n$, though the computations of the Levi form would quickly get very lengthy and would be hard to handle.

Every Stein manifold of dimension $n$ can be realized as a $\textrm{CW}$-complex of dimension at most $n$ (see Andreotti and Frankel \cite{AF}). 
A natural question related to our problem is if one can find regular Stein neighborhoods of a handlebody obtained by attaching a totally real 
handle to a strongly pseudoconvex domain. For results in this directions see the monograph \cite{KK} and the papers by Eliashberg \cite{E}, 
Forstneri\v {c} and Kozak \cite{FHand} and others. We shall not consider this matter here.

\section{Preliminaries}\label{Sb}


A real linear subspace in $\mathbb{C}^n$ is called {\em totally real} if it contains no complex subspace. It is clear that the real dimension of 
a totally real subspace in $\mathbb{C}^n$ is at most $n$.

Now let $M$ and $N$ be two linear totally real subspaces of real dimension $n$ in $\mathbb{C}^n$, intersecting only at the origin. The next lemma 
describes the basic properties of a union of such subspaces $M\cup N$. It is well known and it is not difficult to prove. 
We refer to \cite{Wein} for the proof of the lemma and a short note on linear totally real subspaces in $\mathbb{C}^n$.

\begin{lemma}\label{TR}
Let $M$ and $N$ be two linear totally real subspaces of real dimension $n$ in $\mathbb{C}^n$ and with intersection $M \cap N=\{ 0\}$. Then there exists a 
non-singular complex linear transformation which maps $N$ onto $\mathbb{R}^n\approx (\mathbb{R}\times \{0\})^n\subset \mathbb{C}^n$ and $M$ onto 
$M(A)=(A + i I)\mathbb{R}^n$, where $A$ is a matrix with real entries and such that $i$ is not an eigenvalue of $A$. Moreover, any non-singular real 
matrix $S$ maps $M(A)\cup \mathbb{R}^n$ onto $M(SAS^{-1})\cup \mathbb{R}^n$. 
\end{lemma}

Our goal is to construct Stein neighborhoods of a union of totally real planes $M$ and $N$ in $\mathbb{C}^2$, intersecting only at the origin 
(see Sect. \ref{baza}). It is easy to see that non-singular linear transformations map Stein domains onto Stein domains and totally real subspaces 
onto totally real subspaces. According to Lemma \ref{TR} the general situation thus reduces to the case 
$N=\mathbb{R}^2\approx (\mathbb{R}\times \{0\})^2\subset \mathbb{C}^2$ and $M=(A + i I)\mathbb{R}^2$, where $A$ satisfies one of three conditions 
listed below. (In each case we also add an orthogonal complement $M^{\bot}$ to $M$ and the squared Euclidean distance function $d_M$ to $M$ in 
$\mathbb{C}^2= (\mathbb{R}+ i\mathbb{R})^2\approx \mathbb{R}^4$; they are all given in corresponding real coordinates 
$(x,y,u,v)\approx (x+iy,u+iv)\in \mathbb{C}^2$.)

\begin{enumerate}
\item[{\bf Case 1.}] $A$ is diagonalizable over $\mathbb{R}$, i.e. 
            $A=\left [\begin{array}{c c}
                                    a & 0\\
	                            0 & d
                      \end{array}\right], \, a,d\in\mathbb{R}$,
	    \begin{align}\label{dD}
	                 & M = \Span\{(a,1,0,0),(0,0,d,1)\},\quad M^{\bot}=\Span\{(1,-a,0,0),(0,0,1,-d)\},\\
	                 & d_M(x,y,u,v) =\frac{(u-dv)^2}{1+d^2}+\frac{(x-ay)^2}{1+a^2}.\nonumber
	    \end{align}

\item[{\bf Case 2.}] $A$ has complex eigenvalues (but $i$ is not an eigenvalue), i.e. 
	                    $A=\left [\begin{array}{c c}
                                                       a & -d\\
	                                               d & a
                                      \end{array}\right],\\ 
                            a,d\in\mathbb{R}, \, d\neq 0,\, a^2+(1-d^2)^2\neq 0$,
             \begin{align}\label{dC}
	    	       	     & M=\Span\{(a,1,d,0),(-d,0,a,1)\},\\
	    	       	     & M^{\bot}=\Span\{(0,-d,1,-a),(1,-a,0,d)\},\nonumber \\
	    	       	     & d_M (x,y,u,v) =\frac{(u-dy-av)^2}{1+a^2+d^2}+\frac{(x-ay+dv)^2}{1+a^2+d^2}.\nonumber
	     \end{align}

\item[{\bf Case 3.}] $A$ is non-diagonalizable, i.e.  
	                     $A=\left [\begin{array}{c c}   
	                                            a & d\\
	                                            0 & a
                                       \end{array}\right], \,\, a\in\mathbb{R}, \,d \neq 0$,
	        \begin{align}\label{dJ}
	                        & M=\Span \{(a,1,0,0),(d,0,a,1)\}, \\
	                        & M^{\bot}=\Span\{(0,0,1,-a),(1,-a,\frac{-ad}{1+a^2},\frac{-d}{1+a^2})\}, \nonumber \\
	                        & d_M(x,y,u,v)=\frac{(u-av)^2}{1+a^2}+\frac{((1+a^2)(x-ay)-dau-dv)^2}{(1+a^2)((1+a^2)^2+d^2)} .\nonumber
	        \end{align}               
\end{enumerate}

Our construction of Stein domains involves strictly plurisubharmonic functions and strong pseudoconvexity. Here we recall the basic definitions and 
establish the notation.

Given a $\mathcal{C}^2$-function $\rho$ on a complex manifold $X$, 
we define the {\em Levi form} by
\[
	\mathcal{L}_{(z)}(\rho;\lambda) = \left\langle \partial\overline{\partial} \rho(z), \lambda \wedge \bar \lambda\, \right\rangle, \qquad 
                                           z\in X,\quad \lambda \in T_z^{1,0}X\approx T_zX,
\]
where $T_z^{1,0}X$ is the eigenspace corresponding to the eigenvalue $i$ 
of the underlying almost complex structure operator $J$ on the complexified tangent bundle $\mathbb{C}\otimes_\mathbb{R}TX$. 
In local holomorphic coordinates $z=(z_1,\ldots,z_n)$ we have 
\[
   \mathcal{L}_{(z)}(\rho;\lambda) = \sum_{j,k =1}^n \frac{\partial^2 \rho}{\partial z_j\partial\bar{z}_k }(z) \lambda_j\overline{\lambda}_k, \quad 
                                    \lambda=\sum_{j=1}^n \lambda_j \frac{\partial}{\partial z_j}.
\]
A function $\rho$ is {\em strictly plurisubharmonic} if and only if 
$\mathcal{L}_{(z)}(\rho;\cdot)$ is a positive definite Hermitian quadratic form for all $z\in X$.

Let $\rho\colon \mathbb{C}^n\to\mathbb{R}$ be a $\mathcal{C}^2$ defining function for $\Omega\subset\mathbb{C}^n$, i.e. 
$\Omega=\{z\in\mathbb{C}^n\colon \rho(z)<c\}$ and $b\Omega=\{z\in\mathbb{C}^n\colon \rho(z)=c\}$ for some $c\in \mathbb{R}$. If also 
$d\rho(z)\neq 0$ for every $z\in b\Omega$ we say that $\Omega$ has $\mathcal{C}^2$-boundary.

A domain $\Omega\subset \mathbb{C}^n$ is {\em strongly Levi pseudoconvex} if for every $z\in b\Omega$ the Levi form of $\rho$ is positive in all 
complex tangent directions to the boundary $b\Omega$:
\[
              \mathcal{L}_{(z)}(\rho;\lambda)>0, \quad z\in b\Omega, \quad \lambda\in T_z ^{\mathbb{C}}(b \Omega):= T_z (b \Omega)\cap  i T_z (b \Omega).
\]
If $\rho$ strictly plurisubharmonic in a neighborhood of the boundary $b\Omega$, a domain $\Omega$ is said to be {\em strongly pseudoconvex}.

Throughout this paper $(z_1,z_2)$ will be standard holomorphic coordinates and $(x,y,u,v)$ corresponding real coordinates on $\mathbb{C}^2$ with respect 
to $z_1=x+iy$ and $z_2=u+iv$. Holomorphic and antiholomorphic derivatives are in standard notation denoted by
$\frac{\partial}{\partial z_1}=\frac{1}{2}\left(\frac{\partial}{\partial x}-i\frac{\partial}{\partial y}\right)$,
$\frac{\partial}{\partial \overline{z}_1}=\frac{1}{2}\left(\frac{\partial }{\partial x}+i\frac{\partial}{\partial y}\right)$
or briefly by 
$\frac{\partial \rho}{\partial z_1}=\rho_{z_1}$, $\frac{\partial \rho}{\partial \overline{z}_1}=\rho_{\overline{z}_1}$,
and the same for $\frac{\partial}{\partial z_2}$, $\frac{\partial}{\partial \overline{z}_2}$.

If $\rho$ defines a domain $\Omega\subset\mathbb{C}^2$, we have  
\[
                  T_z^{\mathbb{C}}(b\Omega)=\{(w_1,w_2)\colon\frac{\partial\rho}{\partial z_1}(z)\,w_1+\frac{\partial\rho}{\partial z_2}(z)\,w_2=0\}
\]
and we denote the vector in complex tangent direction to the boundary $b\Omega$ by
\begin{equation}\label{vektorC}
                 \lambda_{\rho}=\left(\frac{\partial\rho}{\partial z_2},-\frac{\partial \rho}{\partial z_1}\right)\in T^{\mathbb{C}}(b\Omega).
\end{equation}
A straightforward calculation then gives
\begin{align}\label{formaC0}
    \mathcal{L}(\rho;\lambda_{\rho}) &= 
                                         \rho_{z_1\overline{z}_1}\rho_{z_2}\overline{\rho}_{z_2}+\rho_{z_2\overline{z}_2}\rho_{z_1}\overline{\rho}_{z_1}
                                         -\rho_{z_2\overline{z}_1}\rho_{z_1}\overline{\rho}_{z_2}-\rho_{z_1\overline{z}_2}\rho_{z_2} \overline{\rho}_{z_1}\\
                                     &=
                                          \rho_{z_1\overline{z}_1}|\rho_{z_2}|^2+\rho_{z_2\overline{z}_2}|\rho_{z_1}|^2
                                          -2\textrm{Re}(\rho_{z_2\overline{z}_1}\rho_{z_1}\overline{\rho}_{z_2}).\nonumber 
\end{align}
In terms of real partial derivatives, we have
\begin{align}\label{formaC}
\mathcal{L}(\rho;\lambda_{\rho}) &=
                                  \frac{1}{16}\left(\frac{\partial^2\rho}{\partial x^2}+\frac{\partial^2\rho}{\partial y^2}\right)
                                  \left(\left(\frac{\partial \rho}{\partial u}\right)^2+\left(\frac{\partial \rho}{\partial v}\right)^2\right) \\
                                 &\quad 
                                  +\frac{1}{16}\left(\frac{\partial^2\rho}{\partial u^2}+\frac{\partial^2 \rho}{\partial v^2}\right)
                                  \left(\left(\frac{\partial \rho}{\partial x}\right)^2+\left(\frac{\partial\rho}{\partial y}\right)^2\right)\nonumber  \\
                                 & \quad
                                  -\frac{1}{8}\left(\frac{\partial^2\rho}{\partial x\partial u}+\frac{\partial^2\rho}{\partial y\partial v}\right)
                                              \left(\frac{\partial\rho}{\partial x}\frac{\partial\rho}{\partial u}
                                                    +\frac{\partial \rho}{\partial y}\frac{\partial \rho}{\partial v}\right)\nonumber \\
			         & \quad
			          +\frac{1}{8}\left(-\frac{\partial^2\rho}{\partial x\partial v}+\frac{\partial^2\rho}{\partial y\partial u}\right)
			                      \left(\frac{\partial\rho}{\partial v} \frac{\partial \rho}{\partial x}
			                            -\frac{\partial\rho}{\partial y}\frac{\partial\rho}{\partial u}\right).\nonumber 
\end{align}

\section{Local construction at the intersection}

In this section we give a local construction of regular Stein neighborhoods near the intersection $M\cap N=\{0\}$ of a union of two 
totally real planes $M\cup N\subset \mathbb{C}^2$. Our goal is to find a function $\rho\colon \mathbb{C}^n\to \mathbb{R}$ satisfying 
the following properties:
\begin{enumerate}
\item \label{pogoj1} $M\cup N=\{\rho=0\}=\{\nabla\rho=0\}$,
\item \label{pogoj2} $\Omega_{\epsilon}=\{\rho<\epsilon\}$ is strongly Levi pseudoconvex for any sufficiently small $\epsilon>0$. 
\end{enumerate}
Observe that in this case the flow of the negative gradient vector field $-\nabla\rho$ gives us a strong deformation retraction of 
$\Omega_{\epsilon}$ to $M\cup N$.

In order to fulfil the conditions (\ref{pogoj1}) and (\ref{pogoj2}) one might take linear combinations of products of squared Euclidean 
distance functions to $M$ and $N$ in $\mathbb{C}^2$ respectively. However, the Levi form of such a function would be a polynomial of high 
degree and therefore very difficult to control. In order to simplify the situation we shall prefer to work with homogeneous polynomials. 
The following lemma is a preparation for our key result Lemma \ref{lema11}.

\begin{lemma}\label{lemamn}
Let $A$, $M$ and $d_M$ be of the form as in (\ref{dD}), (\ref{dC}) or (\ref{dJ}) and let $N=\mathbb{R}^2$ with $d_N(x,y,u,v)=y^2+v^2$. Then the function 
\[
       \rho=d_M ^{\alpha+1}d_N^{\beta}+d_M ^{\alpha}d_N ^{\beta+1}, \quad \alpha,\beta\geq 1
\]
satisfies the following properties:
\begin{enumerate}
\item \label{1} $M\cup N=\{\rho=0\}=\{\nabla\rho=0\}$.
\item \label{3} There exist constants $r>0$ and $\epsilon_0>0$ such that $\rho$ is strictly plurisubharmonic on 
       $(\{d_M<\epsilon_0 \}\cup \{d_N<\epsilon_0 \})\setminus (M\cup N\cup \overline{B}_r)$, where $B_r$ is a ball centered at 
       $0$ and with radius $r$. In addition, for $\alpha=\beta=1$ the Levi form of $\rho$ is positive on a neighborhood of 
       $(M\cup N)\setminus \{0\}$, and for $\alpha,\beta\geq 2$ it vanishes on $M\cup N$. 
\item \label{4} For any $\epsilon>0$ and $\Omega_{\epsilon}=\{\rho<\epsilon\}$ the Levi form of $\rho$ in complex tangent direction 
      to the boundary $b\Omega_{\epsilon}$ is of the form: 
\begin{equation*}
                  \mathcal{L}(\rho;\lambda_{\rho})=\frac{1}{k}\,d_M^{3\alpha-2} d_{N} ^{3\beta-2} P, \qquad \lambda_{\rho} \in T^{\mathbb{C}}(b\Omega_\epsilon),
\end{equation*}
where $k$ is a positive polynomial in the entries of $A$, and $P$ is a homogeneous polynomial of degree $10$ in variables $x,y,u,v$ 
and with coefficients depending polynomially on the entries of $A$.
\end{enumerate}
\end{lemma}

\begin{proof}
Property (\ref{1}) is an immediate consequence of the definition of $\rho$.

Next, we fix $m,n\geq 1$ and for any $\lambda=\sum_{j=1}^{2}\lambda_j\frac{\partial}{\partial z_j} \in T(\mathbb{C}^2)$ we obtain
\begin{align}\label{formaRomn}
           \mathcal{L}(d_M ^{m}d_N^n;\lambda) &=  
                                                 md_M ^{m-1}d_N ^{n} \mathcal{L}( d_M;\lambda)
                                                 +(m-1)m d_M ^{m-2}d_N ^{n}\left|\sum_{j=1}^{2}\frac{\partial d_M}{\partial z_j}\lambda_j\right|^2\\
                                              & \quad
                                                 +2mn d_N ^{n-1}d_M ^{m-1}\textrm{Re}\left(\left(\sum_{j=1}^{2}\frac{\partial d_M}{\partial z_j}\lambda_j\right)
                                                     \left(\sum_{j=1}^{2}\frac{\partial d_N}{\partial \overline{z}_j}\overline{\lambda}_j\right)\right)\nonumber \\
                                              & \quad
                                                      +n d_N ^{n-1}d_M ^{m} \mathcal{L}(d_N;\lambda) 
                                                      +(n-1)n d_N ^{n-2}d_M ^{m}\left|\sum_{j=1}^{2}\frac{\partial d_N}{\partial z_j}\lambda_j\right|^2. \nonumber 
\end{align}

It is well known and also very easy to check that the squared distance functions $d_M$ and $d_N$ respectively to totally real subspaces 
$M$ and $N$ are strictly plurisubharmonic. Moreover, there exists a constant $c>0$ such that 
\[
                \mathcal{L}(d_M;\lambda)\geq c|\lambda|^2, \qquad \mathcal{L}(d_N;\lambda)\geq c|\lambda|^2, \qquad \lambda \in T(\mathbb{C}^2).
\]
For some real constant $b>0$ we also have 
\[
       \left| \left(\sum_{j=1}^{2}\frac{\partial d_M}{\partial z_j}\lambda_j\right)
       \left(\sum_{j=1}^{2}\frac{\partial d_N}{\partial \overline{z}_j}\overline{\lambda}_j\right)\right|
                                                                                                         \leq b \sqrt{d_N d_M} \,\bigl|\lambda\bigr|^2, \qquad 
                                                                                                                                    \lambda \in T(\mathbb{C}^2).
\]
Therefore, if we are sufficiently far away from $N$ and close enough to $M$, but not on $M$, the term 
$\,m\,d_M ^{m-1}d_N ^{n}\,\mathcal{L}( d_M;\lambda)$ in (\ref{formaRomn}) will dominate the third term in (\ref{formaRomn}), and will 
thus make $\mathcal{L}(d_M ^{m}d_N^n;\lambda)$ positive there, for all $\lambda$. Similary, the term $n d_N ^{n-1}d_M ^{m} \mathcal{L}(d_N;\lambda)$ 
makes $\mathcal{L}(d_M ^{m}d_N^n;\lambda)$ positive, provided that we are far away from $M$ and close to $N$, but not on $N$. Hence 
$\rho=d_M ^{\alpha+1}d_N^{\beta}+d_M ^{\alpha}d_N ^{\beta+1}$ satisfies the first part of the statement (\ref{3}). Clearly, since 
$\nabla d_M$ vanishes on $M$ and $\nabla d_N$ vanishes on $N$, the Levi form of $\rho$ is positive on $(M\cup N)\setminus \{0\}$ for 
$\alpha=\beta=1$, and vanishes on $M\cup N$ for $\alpha,\beta\geq 2$. This concludes the proof of (\ref{3}).

To prove (\ref{4}) we need to factor $\mathcal{L}(\rho;\lambda_{\rho})$ (see (\ref{formaC0})) into a product of 
$d_M^{3\alpha-2} d_{N} ^{3\beta-2}$ and a polynomial in variables $x,y,u,v$, and with coefficients depending on the entries of $A$. Here we have
\begin{equation}\label{lambdaRoMN}
                            \lambda_{\rho}=\Bigl((\alpha+1) d_N^{\beta}d_M^{\alpha}+\alpha d_N^{\beta+1}d_M^{\alpha-1}\Bigr)\lambda_{d_M}
                                           +\Bigl((\beta+1)d_N^{\beta}d_M^{\alpha}+ \beta d_N^{\beta-1}d_M^{\alpha+1}\Bigr)\lambda_{d_N}.
\end{equation}

Firstly, since $\sum_{j=1}^{2}\frac{\partial d_M}{\partial z_j}{\lambda_{d_M}}_j=0$ and 
$\sum_{j=1}^{2}\frac{\partial d_N}{\partial z_j}{\lambda_{d_N}}_j=0$, we can clearly factor 
$\sum_{j=1}^{2}\frac{\partial d_M}{\partial z_j}{\lambda_{\rho}}_j$ and $\sum_{j=1}^{2}\frac{\partial d_N}{\partial z_j}{\lambda_{\rho}}_j$ 
respectively into a product of $d_M^{\alpha}$ or $d_N^{\beta}$ and a polynomial in variables $x,y,u,v$.

Next, we observe that $d_M^{2\alpha-2}$ (respectively $d_N^{2\beta-2}$) factor out of $\mathcal{L}(d_M;\lambda_{\rho})$ or 
$\mathcal{L}(d_N;\lambda_{\rho})$, trivially. An easy and straightforward computation by using (\ref{formaC}) shows further 
that $\mathcal{L}(d_N;\lambda_{d_N})=\frac{1}{2}d_N$, while if $d_M$ is as in (\ref{dD}), (\ref{dC}) and (\ref{dJ}) respectively, 
we obtain $\mathcal{L}(d_M;\lambda_{d_M})=\frac{1}{2}d_M$, or we get $\mathcal{L}(d_M;\lambda_{d_M})=\frac{((1+a^2+d^2)^2-4d^2)}{2(1+a^2+d^2)^{2}}d_M$ 
and $\mathcal{L}(d_M;\lambda_{d_M})=\frac{(1+a^2)^2}{2((1+a^2)^2+d^2)}d_M$. Hence $d_M^{2\alpha-1}$ and $d_N^{2\beta-1}$ factor out 
of $\mathcal{L}(d_M;\lambda_{\rho})$ and $\mathcal{L}(d_N;\lambda_{\rho})$, respectively.

From (\ref{formaRomn}) applied in the cases $m=\alpha+1$, $n=\beta$ and $m=\alpha$, $n=\beta+1$ respectively, it now follows immediately that 
$\mathcal{L}(\rho;\lambda_{\rho})$ is factored into a product of $d_M^{3\alpha-2}d_N^{3\beta-2}$ and a polynomial in variables $x,y,u,v$. 
However, there are terms of $\mathcal{L}(\rho;\lambda_{\rho})$ which include $d_M^{3\alpha+3}$ as a factor. For $d_M$ as in case (\ref{dD}) 
we then obtain that $\mathcal{L}(\rho;\lambda_{\rho})$ is of the form
\begin{equation}\label{case21k}
                                  \mathcal{L}(\rho;\lambda_{\rho})=\frac{1}{(1+a^2)^5(1+d^2)^5}d_M^{3\alpha-2} d_{N} ^{3\beta-2} P,
\end{equation}
where $P$ is a homogeneous polynomial of degree $10$ in variables $x,y,u,v$ and the coefficients of $P$ are polynomials in variables $a$ and $d$. 
If $d_M$ is of the form (\ref{dC}) or (\ref{dJ}) respectively, we have
\begin{equation}\label{case22k}
                                 \mathcal{L}(\rho;\lambda_{\rho})=\frac{1}{(1+a^2+d^2)^5}d_M^{3\alpha-2} d_{N} ^{3\beta-2} P,
\end{equation}
and
\begin{equation}\label{case23k}
                                 \mathcal{L}(\rho;\lambda_{\rho})=\frac{1}{(1+a^2)^5((1+a^2)^2+d^2)^5}d_M^{3\alpha-2} d_{N} ^{3\beta-2} P,
\end{equation}
where $P$ again has all the properties required. This concludes the proof of the lemma. 
\end{proof}

We note here that by choosing suitable substitutions, it is also possible to compute explicitly the polynomial $P$ in Lemma \ref{lemamn} (\ref{4}), 
but on the other hand this might involve very long expansions of polynomials. (See also the proof of Lemma \ref{lema11} for this approach in 
the special case $A=0$.)

Before stating a key lemma of our construction we prove the following argument on homogeneous polynomials.

\begin{lemma}\label{lemaPR}
Let $Q,R\in \mathbb{R}[x_1,x_2,\ldots,x_m]$ be real homogeneous polynomials in $m$ variables and of even degree $s$. Assume further that $Q$
is vanishing at the origin and is positive elsewhere. Then for any sufficiently small constant $\epsilon_0 >0$, it follows that 
$Q\geq \epsilon_0\cdot |R|$, with equality precisely at the origin.
\end{lemma}

\begin{proof}
By $||x||=\sqrt{x_1^2+x_2^2+\ldots+x_m^2}$ we denote the standard Euclidean norm of $x=(x_1,x_2,\ldots,x_m)\in \mathbb{R}^m$.

Since $Q$ is vanishing at the origin and is positive elsewhere, there exists a constant $c>0$ such that $Q(x)\geq c$ for 
all $x$ on the unit sphere, i.e. $||x||=1$. Also, there exists a constant $C>0$ such that $|R(x)|\leq C$ for any $x$ on the unit sphere.

However, homogeneous polynomials are uniquely determined by their values on a unit sphere. Thus we have
$Q(x)\geq c||x||^s$ and $|R(x)|\leq C||x||^s$ for any $x$, and with equalities precisely at the origin. The conclusion of the lemma now clearly follows.
\end{proof}

The following lemma is essential in the proof of Theorem \ref{izrek}, where we construct Stein neighborhoods.

\begin{lemma}\label{lema11}
Let $A$, $M$, $d_M$, $N$ and $d_N$ be as in Lemma \ref{lemamn} and let the function $\rho$ be defined as 
\[
     \rho=d_M ^{2}d_N+d_M d_N ^{2}.
\]
If the entries of $A$ are sufficiently close to zero, then for any $\epsilon>0$ the sublevel set 
$\Omega_{\epsilon}=\{\rho <\epsilon\}$ is strongly Levi pseudoconvex.
\end{lemma}

\begin{proof}
By Lemma \ref{lemamn} the Levi form of $\rho=d_M ^{2}d_N+d_M d_N ^{2}$ in complex tangent direction 
$\lambda_{\rho}$ (see (\ref{vektorC})) to the boundary of its sublevel set $\Omega_{\epsilon}=\{\rho<\epsilon\}$ 
is of the form
\begin{equation}\label{LF}
                            \mathcal{L}(\rho;\lambda_{\rho})=\frac{1}{k}\,d_M d_{N} P, \qquad \lambda_{\rho} \in T^{\mathbb{C}}(b \Omega_{\epsilon}),
\end{equation}
where $k$ is a positive polynomial in variables $a$, $d$ (see (\ref{case21k}), (\ref{case22k}) or (\ref{case23k}), 
and $P$ is a homogeneous polynomial of degree $10$ in variables $x,y,u,v$. Furthermore, the coefficients of 
the polynomial $P$ are polynomials in variables $a$ and $d$; these are the entries of $A$ (see (\ref{dD}), (\ref{dC}) or (\ref{dJ})).

We now write $P$ as a sum of two polynomials in variables $x,y,u,v$:
\begin{equation}\label{ulema}
                              P=Q+R,
\end{equation} 
where the coefficients of $Q$ do not depend on $a$ or $d$, and the coefficients of $R$ are polynomials in 
variables $a,d$, and they are in addition without constant term.

Observe further that for $a=d=0$, the Levi form in (\ref{LF}) is equal to the product $(x^2+u^2)(y^2+v^2)\,Q$. 
On the other it is precisely equal to the Levi form of the function 
\[
           \rho_0(x,y,u,v)=(x^2+u^2)^2(v^2+y^2)+(x^2+u^2)(v^2+y^2)^2
\]
in complex tangent direction $\lambda_{\rho_0}$ to the boundary of its sublevel set, which means that 
\begin{equation}\label{LeviFormaRho0}
                                         \mathcal{L}(\rho_0;\lambda_{\rho_0})=(x^2+u^2)(y^2+v^2)\,Q.
\end{equation}

In order to be able to simplify the computation of the Levi form of $\rho_0$ by using (\ref{formaC0}) and 
(\ref{formaRomn}), we now need to substitute certain expressions by suitable new variables. We introduce the notation
\begin{equation}\label{oVZ}
                                         V=v^2+y^2, \qquad Z=u^2+x^2, \qquad \omega=V+Z,
\end{equation} 
respectively. With the new notation, we apply formula (\ref{formaRomn}) for $d_M=Z$, $d_N=V$ in the cases 
$m=2$, $n=1$ and $m=1$, $n=2$. After adding the obtained expressions and slightly regrouping like terms, we get 
\begin{align}\label{formaRo0VZ}
        \mathcal{L}(\rho_0;\lambda) &=
                                      (2Z V + V^2) \mathcal{L}(Z;\lambda)+(Z^2+2 V Z) \mathcal{L}(V;\lambda) \\
                                    & \quad
                                      +(4 Z +4 V)\,\textrm{Re}\left(\left(\sum_{j=1}^{2}\frac{\partial Z}{\partial z_j}\lambda_j\right)
                                                        \left(\sum_{j=1}^{2}\frac{\partial V}{\partial \overline{z}_j}\overline{\lambda}_j\right)\right)\nonumber \\
                                    & \quad
                                      +2 V \left|\sum_{j=1}^{2}\frac{\partial Z}{\partial z_j}\lambda_j\right|^2
                                      +2 Z\left|\sum_{j=1}^{2}\frac{\partial V}{\partial z_j}\lambda_j\right|^2. \nonumber 
\end{align}

Next, observe that  
\[
            \frac{\partial Z}{\partial z_1}=x, \qquad \frac{\partial Z}{\partial z_2}=u, \qquad 
            \frac{\partial V}{\partial z_1}=-iy, \qquad \frac{\partial V}{\partial z_2}=-iv,
\]
and by (\ref{vektorC}) we also have
\begin{equation}\label{lambdaZV}
                         \lambda_Z=(u,-x), \qquad \lambda_V=(-iv,iy).
\end{equation}
By taking $\alpha=\beta=1$ and $d_M=Z$, $d_N=V$ in (\ref{lambdaRoMN}), we further obtain that
\begin{equation}\label{lambdaRo0ZV}
      \lambda_{\rho_0}=(Z+\omega)V\lambda_Z+(V+\omega)Z\lambda_V.
\end{equation}

An easy computation gives us
\begin{equation}\label{LeviVZ}
\mathcal{L}(V;\lambda)=\mathcal{L}(Z;\lambda)=\frac{1}{2}|\lambda|^2, \qquad \lambda \in T(\mathbb{C}^2).
\end{equation}
By combining (\ref{oVZ}), (\ref{lambdaZV}), (\ref{lambdaRo0ZV}), (\ref{LeviVZ}), and regrouping the terms, we now get
\begin{align}
\mathcal{L}(V;\lambda_{\rho_0}) &= \frac{1}{2}\Bigl((Z+\omega)^2V^2(u^2+x^2)+(V+\omega)^2Z^2(v^2+y^2)\Bigr)\\
                                &= \frac{1}{2}\Bigl((Z+\omega)^2V^2Z+(V+\omega)^2Z^2V\Bigr)\nonumber\\
                                &= \frac{1}{2}V Z \Bigl(VZ(Z+V)+4\omega VZ+\omega^2 (V+Z)\Bigr)\nonumber\\
                                &= \frac{1}{2}V Z \omega(5V Z+\omega^2).\nonumber
\end{align}

It is also easy to calculate
\begin{equation}\label{sumlambda}
   \sum_{j=1}^{2}\frac{\partial Z}{\partial z_j}{\lambda_{\rho_0}}_j=-i(V+\omega)Z\Delta, \qquad 
   \sum_{j=1}^{2}\frac{\partial V}{\partial z_j}{\lambda_{\rho_0}}_j=i(Z+\omega)V\Delta,
\end{equation}
where we denoted $\Delta=x v- u y$. By using (\ref{oVZ}) and (\ref{sumlambda}) we can regroup and 
simplify the sum of the last three terms in (\ref{formaRo0VZ}). We obtain
\begin{align*}
         &  -4\omega (V+\omega)(Z+\omega)V Z\Delta^2+2V\bigl((V+\omega)Z\Delta\bigr)^2+2Z\bigl((Z+\omega)V\Delta\bigr)^2\\
         &= -2V Z \Delta^2\Bigl( 2\omega (V+\omega)(Z+\omega)-\left(Z(V+\omega)^2+V(Z+\omega)^2\right)\Bigr)\\
         &= -2V Z \Delta^2\Bigl( 2\omega \left(VZ+\omega(V+Z)+\omega^2\right)-\left(Z V (V+Z)+4\omega V Z+\omega^2 (Z+V)\right)\Bigr)\\
         &= -2 V Z \Delta^2\Bigl( 2\omega (V Z+2\omega ^2)-\left(5\omega V Z+\omega^3\right)\Bigr)\\
         &= -6 V Z \omega \Delta^2(\omega^2-V Z).
\end{align*}

Finally, we have
\begin{align}
         \mathcal{L}(\rho_0;\lambda_{\rho_0}) &= \frac{1}{2}V Z \omega(5V Z+\omega^2)(4Z V + V^2+Z^2)-6 V Z \omega \Delta^2\left(\omega^2-V Z\right)\nonumber\\
                                              &= \frac{1}{2}V Z \omega\Bigl((5V Z+\omega^2)(2Z V + \omega^2)-12\Delta^2\left(\omega^2-V Z\right)\Bigr).\nonumber
\end{align}
After substituting $\omega,V,Z,\Delta$ in the above expression back by the variables $x,y,u,v$, and comparing it 
to (\ref{LeviFormaRho0}), we further obtain the factorization
\begin{equation}\label{P1}
                        Q(x,y,u,v)=\frac{1}{2}(x^2+y^2+u^2+v^2)P_0(x,y,u,v),
\end{equation}
where $P_0$ is a homogeneous polynomial of degree $8$ in variables $x,y,u,v$.

Next, we observe the sign of polynomial $P_0$. We use the Cauchy-Schwarz inequality 
\[
             \Delta ^2=(x v- y u)^2\leq (x^2+u^2)(y^2+v^2)=V Z
\]
in order to see that
\begin{align*}
          P_0 &=     (5V Z+\omega^2)(2Z V + \omega^2)-12\Delta^2\left(\omega^2-V Z\right)\\
              &\geq  22(V Z)^2-5 (V Z)\omega^2+\omega^4\nonumber \\
              &\geq  22\left(V Z-\frac{5}{44}\omega^2\right)^2+\frac{63}{88}\omega^4.\nonumber
\end{align*}
This proves that $P_0$ and hence also $Q$ (see (\ref{P1})), both vanish at the origin and are positive 
everywhere else. Moreover, we obtain that
\begin{equation}\label{ocenaP0}
                      P_0(x,y,u,v)\geq \frac{63}{88}(x^2+y^2+u^2+v^2)^4.
\end{equation}

We now show that polynomial $P$ in (\ref{ulema}) vanishes at the origin and is positive elsewhere, provided 
that the entries $a,d$ of the matrix $A$ are chosen sufficiently small. Recall that the polynomial $R$ (see (\ref{ulema})) is of the form 
\begin{equation}\label{cleniR}
             R(x,y,u,v)= \sum_{|\alpha|=10} S_{\alpha}(a,d)\,\,x ^{\alpha_1} y ^{\alpha_2}u ^{\alpha_3}v ^{\alpha_4},\qquad 
\end{equation}
where $\alpha=(\alpha_1,\ldots,\alpha_4)$ is a multiindex, and $S_{\alpha}$ is a polynomial in variables $a$, $d$. 
Remember also that all $S_{\alpha}$ are without constant terms and hence we have $S_{\alpha}(0,0)=0$. We denote by 
$N_0$ the number of terms of the polynomial $R$. Since $Q$ is a homogeneous polynomial of degree $10$ 
(see (\ref{P1})), which is positive everywhere except at the origin, we can use Lemma \ref{lemaPR} to get a 
constant $\epsilon_0>0$ such that 
\begin{equation}\label{admajhna}
             \frac{1}{N_0}Q \geq\epsilon_0 |x ^{\alpha_1} y ^{\alpha_2}u ^{\alpha_3}v ^{\alpha_4}|, \qquad  \alpha=(\alpha_1,\ldots,\alpha_4), \quad |\alpha|=10,
\end{equation}
where equality holds precisely at the origin. By continuity argument, we can also have $|S_{\alpha}(a,d)|<\epsilon_0$ for all $a$, $d$ small enough, and this estimate is uniform for all coefficients $S_{\alpha}$ of polynomial $R$. It then follows from (\ref{admajhna}) that for all sufficiently small $a$ and $d$, we have $Q\geq|R|$, with equality precisely at the origin. This implies that polynomial $P$ vanishes at the origin and is positive elsewhere. Finally, the Levi form of $\rho$ (see (\ref{LF})) is then positive in complex tangent direction to $b\Omega_{\epsilon}$ for any $\epsilon$. This completes the proof.
\end{proof}

\begin{remark}\label{ocenae}
By analyzing the part of the proof of Lemma \ref{lema11} where Lemma \ref{lemaPR} was applied, we can tell how 
small the entries of the matrix $A$ in the assumption of Lemma \ref{lema11} can be. By combining (\ref{P1}) 
and (\ref{ocenaP0}) we see that
\begin{equation}\label{ocenaQ}
                  Q(x,y,u,v)\geq \frac{63}{176}(x^2+y^2+u^2+v^2)^5.
\end{equation}
As we expect the entries $a,d$ of the matrix $A$ to be smaller than one, we can roughly estimate 
the coefficients $S_{\alpha}$ of the polynomial $R$ (see (\ref{ulema}) and (\ref{cleniR})) by  
$|S_{\alpha}(a,d)|\leq N_{\alpha}\max\{|a|,|d|\}$, where $N_{\alpha}$ denotes the sum of moduli of 
coefficients of the polynomial $S_{\alpha}$. Thus we get
\begin{equation}\label{ocenaM}
   N_1\,N_0\max\{|a|,|d|\} \left(x^{10}+y^{10}+u^{10}+v^{10}\right)\geq |R(x,y,u,v)|,
\end{equation}
where $N_1=\max_{|\alpha|=10}N_{\alpha}$ and $N_0$ is the number of terms of $R$. It follows 
from (\ref{ocenaQ}) and (\ref{ocenaM}), that for any $|a|,|d|< \frac{63}{176\,N_0 N_1}$ we 
have $Q\geq|R|$, with equality precisely at the origin. 
\end{remark}

\begin{remark}
The conclusion of Lemma \ref{lema11} holds, for instance, also for the function $d_M ^2d_N^2+d_M d_N ^{3}$. 
One might expect to prove even more. But on the other hand it is not clear at the moment how that would 
improve the conclusion of the lemma for bigger entries of $A$. 
\end{remark}

\section{Regular Stein neighborhoods of the union of totally real planes}\label{baza}

A system of open Stein neighborhoods $\{\Omega_{\epsilon}\}_{\epsilon\in (0,1)}$ of a set $S$ 
in a complex manifold $X$ is called a {\em regular}, if for every $\epsilon \in (0,1)$ we have
\begin{enumerate}
\item $\Omega_{\epsilon}=\cup_{t<\epsilon}\Omega_t, \qquad \overline{\Omega}_{\epsilon}=\cap_{t>\epsilon}\Omega_{t}$,
\item $S=\cap_{\epsilon\in (0,1)}\Omega_{\epsilon}$ is a strong deformation retract of every $\Omega_{\epsilon}$ with $\epsilon\in (0,1)$.
\end{enumerate}

\begin{theorem}\label{izrek}
Let $A$ be a real $2\times 2$ matrix such that $A-iI$ is invertible. Further, let 
$M=(A+iI)\mathbb{R}^2$ and $N=\mathbb{R}^2$. If the entries of $A$ are sufficiently 
small, then the union $M\cup N$ has a regular system of strongly pseudoconvex Stein 
neighborhoods in $\mathbb{C}^2$. Moreover, away from the origin the neighborhoods coincide 
with sublevel sets of the squared Euclidean distance functions to $M$ and $N$, respectively. 
\end{theorem}

As noted in Sect. \ref{Sb}, the general case of union of two totally real planes intersecting 
at the origin reduces to the situation described in the Theorem \ref{izrek}. Furthermore, 
we may assume that $M$ is of the form as in one of the three cases (\ref{dD}), (\ref{dC}) or (\ref{dJ}).

\begin{proof}
Lemma \ref{lema11} furnishes a function $\rho=d_M ^{2}d_N+d_M d_N ^{2}$, where $d_M$ and $d_N$ 
respectively are squared Euclidean distance functions to $M$ and $N$ in $\mathbb{C}^2$. For 
any $\epsilon>0$, a domain $\Omega_{\epsilon}=\{\rho <\epsilon\}$ is strongly Levi pseudoconvex. 
Also, the Levi form of $\rho$ is positive on $(M\cup N)\setminus\{0\}$ and we have 
$\{\rho=0\}=\{\nabla \rho=0\}=M \cup N$ (see Lemma \ref{lemamn}).

We proceed by patching $\rho$ away from the origin with the squared distance functions. First we 
choose open balls $B_r$ and $B_{2r}$ respectively, centered at $0$ and  with radii $r$ and $2r$. 
Next, for any $\epsilon >0$ we set 
\[
     T_{\epsilon,M}=\{z\in\mathbb{C}^2\setminus \overline{B}_{r}\colon d_M(z)<\epsilon\}, \qquad  
     T_{\epsilon,N}=\{z\in\mathbb{C}^2\setminus\overline{B}_{r}\colon d_N(z)<\epsilon\}
\]
and observe that for $\epsilon$ small enough the set $T_{\epsilon}=T_{\epsilon,M}\cup T_{\epsilon,N}$ 
is a disjoint union. We now glue $\rho$ on $B_{2r}$ with the restrictions 
$\rho_M=d_M|_{T_{\epsilon,M}}$ and $\rho_N=d_N|_{T_{\epsilon,N}}$:
\[
      \rho_0(z)=\theta(z)\rho(z)+\bigl(1-\theta(z)\bigr)\rho_M(z)+\bigl(1-\theta(z)\bigr)\rho_N(z),\qquad z\in B_{2r}\cup T_{\epsilon}.
\]
Here $\theta$ is a smooth cut-off function, which is supported on $B_{2r}$ and equals one on $B_r$. To be precise, we have 
$\theta=\chi(|z_1|^2+|z_2|^2)$, where $\chi$ is another suitable cut-off function with $\chi(t)=1$ for $t\leq r$ 
and $\chi(t)=0$ for $t\geq 2r$. Observe that $\rho_0$ coincides with $\rho$ on $\overline{B}_r$ and with $d_M$ or 
$d_N$ respectively on $T_{\epsilon,M}\setminus B_{2r}$ and $T_{\epsilon,N}\setminus B_{2r}$.

It is immediate that $\{\rho_0=0\}=M\cup N$ and that $\nabla\rho_0$ is vanishing on $M\cup N$. On 
$(B_{2r}\setminus \overline{B}_r)\setminus (M\cup N)$, but close to $M\cup N$, we have $\nabla\theta$ 
near to tangent directions to $M\cup N$, and $\nabla\rho_M$ or $\nabla\rho_N$ respectively are near to 
normal directions to $M$ and $N$. After possibly choosing $\epsilon$ smaller and shrinking $T_{\epsilon}$, 
we get $\{\nabla\rho_0=0\}=M\cup N$. Finally, the flow of the negative gradient vector field 
$-\nabla \rho_0$ gives us a deformation retraction of $\Omega_{\epsilon}=\{\rho_0<\epsilon\}$ onto 
$M\cup N$ for every $\epsilon$ small enough.

It remains to verify that the sublevel set $\Omega_{\epsilon}$ is indeed Stein, provided that 
$\epsilon$ is chosen small enough. Since $\rho$, $d_M$, $d_N$ and their gradients all vanish 
on $M\cup N$, this implies that for $z\in M\cup N$ and any $\lambda\in T_z(\mathbb{C}^2)$ we have 
\[
      \mathcal{L}_{(z)}(\rho_0;\lambda)=\theta (z)\mathcal{L}_{(z)}(\rho;\lambda)+\bigl(1-\theta(z)\bigr)\mathcal{L}_{(z)}(\rho_M;\lambda)
                                                                                 +\bigl(1-\theta(z)\bigr)\mathcal{L}_{(z)}(\rho_N;\lambda).
\]
The Levi form of $\rho_0$ is thus positive on $(M \cup N)\setminus \{0\}$. By choosing $\epsilon $ small 
enough, it is then positive on $\Omega_{\epsilon}\setminus B_r$. Furthermore, as $\rho_0$ coincides 
with $\rho$ on $B_r$, the Levi form of $\rho_0$ is positive in complex tangent direction to 
$b\Omega_{\epsilon}$ (by Lemma \ref{lema11}).

We now use a standard argument to get a strictly plurisubharmonic function in all directions also 
on $b\Omega_{\epsilon}\cap B_r$. Set a new defining function for $\Omega_{\epsilon}$:  
\begin{equation}\label{popravi}
                \tilde{\rho}(z)=\bigl(\rho_0(z)-\epsilon\bigr) e^{C(\rho_0(z)-\epsilon)},
\end{equation}
where $C$ is a large constant (to be chosen). By computation we get
\[
  \mathcal{L}_{(z)}(\tilde{\rho};\lambda)=\mathcal{L}_{(z)}(\rho_0;\lambda)+2C\Bigl|\sum_{j=1} ^2 \frac{\partial\rho_0}{\partial z_j}(z)\lambda_j\Bigr|^2, \quad 
                                           z\in b\Omega_{\epsilon},\, \lambda=\sum_{j=1}^{2}\lambda_j\frac{\partial}{\partial z_j} \in T_z(\mathbb{C}^2).
\]
After taking $C$ large enough the Levi form of $\tilde{\rho}$ becomes positive in all 
directions on $b\Omega_{\epsilon}$. This proves strong pseudoconvexity of $\Omega_{\epsilon}$. 
Since the restrictions of plurisubharmonic functions to analytic sets are plurisubharmonic 
and must satisfy the maximum principle (see \cite{GRpluri}), we cannot have any compact 
analytic subset of positive dimension in $\mathbb{C}^2$. As $\Omega_{\epsilon}\subset \mathbb{C}^2$ 
is strongly pseudoconvex, it is then Stein by a result of Grauert (see \cite[Proposition 5]{lit9}). 
This completes the proof.
\end{proof}

\begin{remark}
The assumption of taking sufficiently small entries of $A$ in Theorem \ref{izrek} is essential 
and enables the application of Lemma \ref{lema11} in the proof; see Remark \ref{ocenae} for the 
estimate how small the entries of $A$ can be.
\end{remark}

Lemma \ref{lema11} can also be applied to give an extension of a result on certain 
clo\-sed real surfaces immersed into a complex surface (\cite[Theorem 2.2]{lit1n} and \cite[Theorem 2]{lit2n}).

\begin{trditev}
Let $\pi\colon S\to X$ be an smooth immersion of a closed real surface into a Stein surface satisfying the following properties:
\begin{enumerate}
	\item $\pi$ has only transverse double points (no multiple points) $p_1,\ldots p_k$, 
	      and in a neighborhood of each double point $p_j$ ($j\in\{1,\ldots,k\}$), there 
	      exist holomorphic coordinates $\psi_j\colon U_j\to V_j\subset\mathbb{C}^2$ such 
	      that $\psi_j(\tilde{S}\cap U_j)=(\mathbb{R}^2\cup M_j)\cap V_j$, $\psi_j(p_j)=0$, 
	      where $\tilde{S}=\pi (S)$ and $M_j=(A_j+iI)\mathbb{R}^2$ with $A_j-iI$ invertible,
	\item $\pi$ has finitely many complex points $p_{k+1}, \ldots,p_m$, which are flat hyperbolic.
\end{enumerate}
If the entries of $A_j$ for all $j\in\{1,\ldots,k\}$ are sufficiently close to zero, then $\tilde{S}$ has 
a regular strongly pseudoconvex Stein neighborhood basis in $X$.
\end{trditev}

The proofs given in \cite[Theorem 2.2]{lit1n} and \cite[Theorem 2]{lit2n}) apply mutatis 
mutandis to our situation. For the sake of completeness we sketch the proof.

\begin{proof} 
By Lemma \ref{lema11} for every $j\in\{1,\ldots k\}$ there exists a smooth non-negative 
function $\rho_j\colon V_j \to \mathbb{R}$, which is strictly plurisubharmonic away from 
the origin and its sublevel sets $\{\rho_j<\epsilon\}$ are strongly Levi pseudoconvex. 
Furthermore, we have $\{\rho_j=0\}=\{\nabla \rho_j=0\}=(\mathbb{R}^2\cup M_j)\cap V_j$ 
(see also Lemma \ref{lemamn}). Next we set $\varphi_j=\rho_j\circ \psi_j\colon U_j\to \mathbb{R}$ 
and observe that $\varphi_j$ inherits the obove properties from $\rho_j$.

By \cite[Lemma 8]{lit2n} for every $j\in \{k+1,\ldots,m\}$ there exists a small neighborhood 
$U_j$ of a point $p_j$ and a smooth non-negative function $\varphi\colon U_j\to \mathbb{R}$, 
which is strictly plurisubharmonic on $U_j\setminus \{p_j\}$ and such that 
$\{\varphi_j=0\}=\{\nabla\varphi_j=0\}=\tilde{S}\cap U_j$.

Further, let $\varphi_0=d_{\tilde{S}}$ and $d_p$ respectively be the squared distance functions 
to $\tilde{S}$ or to $p\in \tilde{S}$ in $X$, with respect to some Riemannian metric on $X$. It 
is well known that the squared distance function to a smooth totally real submanifold is strictly 
plurisubharmonic in a neighborhood of the submanifold (see e.g. \cite[Proposition 2]{lit2n} or 
\cite[Proposition 4.1]{lit27}). Therefore $\varphi_0$ is strictly plurisubharmonic in some open 
neighborhood $U_0$ of $\tilde{S}\setminus \{p_1,\ldots,p_m\}$.

We now patch functions $\varphi_j$ for all $j\in \{0,1,\ldots,m\}$. First, denote 
$U=\cup_{j=0} ^{m} U_j$ and let $r\colon U\to \tilde{S}$ be a map defined as 
$r(z)=p$ if $d_{\tilde{S}}(z)=d_{p}(z)$. The map $r$ is well defined and smooth, 
provided that the sets $U_j$ are chosen small enough. Next, we choose a partition of 
unity $\{\theta_j\}_{0\leq j\leq m}$ subordinated to $\{U_j\cap \tilde{S}\}_{0\leq j\leq m}$, 
and such that for every $j\in\{1,\ldots,m\}$ the function $\theta_j$ equals one near the 
point $p_j$. Finally, we define
\[
               \rho(z)=\sum_{j=0} ^m \theta_j\bigl(r(z)\bigr) \varphi_j(z), \qquad z \in U.
\]
We see that $\tilde{S}=\{\rho=0\}$ and $\nabla\rho(z)=\sum_{j=0}^m \theta_j\bigl(r(z)\bigr) \nabla\varphi_j(z)$ 
for all $z\in U$, thus we further have
\[
        \mathcal{L}_{(p)}(\rho;\lambda)=\sum_{j=0} ^m \theta_j(p)\mathcal{L}_{(p)}(\varphi_j;\lambda), \qquad 
                                         p\in \tilde{S}, \quad \lambda \in T_p(U).
\]
After shrinking $U$ we obtain that $\{\nabla\rho=0\}=\tilde{S}$ and $\rho$ is 
strictly plurisubharmonic away from the points $p_{1}, \ldots, p_m$.

It is left to show that the sublevel sets $\Omega_{\epsilon}=\{\rho<\epsilon\}$ are Stein 
domains. Since $\rho$ coincides with $\varphi_j$ near $p_j$ for every $j\in\{1,\ldots,m\}$, 
the sublevel sets $\Omega_{\epsilon}$ are then strongly Levi pseudoconvex near $p_j$. For a 
given $\epsilon$ we can in a similar way as in  the proof of Theorem \ref{izrek} (see (\ref{popravi})) 
choose a positive constant $C$ such that $\tilde{\rho}(z)=\bigl(\rho(z)-\epsilon\bigr) e^{C(\rho(z)-\epsilon)}$ 
is a defining function for $\Omega_{\epsilon}$ and such that $\tilde{\rho}$ is strictly plurisubharmonic 
on $b\Omega_{\epsilon}$. The function $\tilde{\rho}$ might not be strictly plurisubharmonic only near 
the points $p_1,\ldots, p_m$. Since $X$ is Stein we globally have a strictly plurisubharmonic function 
and by standard cutting and patching techniques (see i.e. \cite{lit6}) we obtain a strictly plurisubharmonic 
exhaustion function for $\Omega_{\epsilon}$. By Grauert's theorem \cite[Theorem 2]{lit9} a domain 
$\Omega_{\epsilon}$ is then Stein.
\end{proof}


\bibliographystyle{amsplain}

\end{document}